\documentclass{amsart}

\newtheorem{theorem}{Theorem}[section]
\newtheorem{lemma}[theorem]{Lemma}

\newtheorem{corollary}[theorem]{Corollary}

\theoremstyle{definition}
\newtheorem{example}[theorem]{Example}

\theoremstyle{remark}
\newtheorem{remark}[theorem]{Remark}

\newcommand{\slind}[1]{\index{#1}{\bf #1}}
\newcommand{\gloss}[1]{#1\glossary{\protect #1}}
\newcommand{\salmacite}[1]{\cite{#1}}

\newcommand{\n}{\par\noindent}

\newcommand{\sn}{\par\smallskip\noindent}
\newcommand{\mn}{\par\medskip\noindent}
\newcommand{\bn}{\par\bigskip\noindent}

\newcommand{\ra}{\rightarrow}

\newcommand{\bbox}[1]{\makebox(0,0){\rule[-2ex]{0ex}{5.4ex}#1}}

\newcommand{\cal}{\mathcal}
\newcommand{\lv}{\mathbb}
\newcommand{\N}{\lv N}
\newcommand{\Q}{\lv Q}
\newcommand{\R}{\lv R}
\newcommand{\Z}{\lv Z}

\begin{document}            

\title{The valuation difference rank of a quasi-ordered difference field}
\author{Salma Kuhlmann}
\address{Universit\"at Konstanz\\ FB Mathematik und Statistik\\
78457 Konstanz, Germany}
\email{salma.kuhlmann@uni-konstanz.de}
\author{Micka\"{e}l Matusinski}
\address{IMB, Universit\'e Bordeaux 1, 33405 Talence, France}
\email{mmatusin@math.u-bordeaux1.fr}
\author{Fran\c coise Point}
\address{Institut de math{\'e}matique, Le Pentagone\\
Universit{\'e} de Mons\\
B-7000 Mons, Belgium}
\email{Francoise.Point@umons.ac.be}
\thanks{Supported by a
 Research in Paris grant from Institut Henri Poincar\'e,  Konstanz University, Bordeaux 1 University and the Fonds de la Recherche Scientifique FNRS-FRS}
\maketitle



\begin{abstract}
There are several equivalent characterizations of the valuation rank
of an ordered or valued field. In this
paper, we extend the theory to the case of an ordered or valued {\it difference}
field (that is, ordered or valued field endowed with a compatible field automorphism). We introduce the notion of {\it difference rank}. To treat simultaneously the cases of ordered and valued fields, we consider quasi-ordered fields. We
characterize the difference rank as the quotient modulo the
equivalence relation naturally induced by the automorphism (which
encodes its growth rate). In analogy to the theory of convex
valuations, we prove that any linearly ordered set can be realized
as the difference rank of a maximally valued quasi-ordered difference field. As an application, we show that for every regular uncountable cardinal  $\kappa$ such that  $\kappa= \kappa^{< \kappa}$, there are $2^{\kappa}$ pairwise non-isomorphic quasi-ordered difference fields of cardinality $\kappa$, but all isomorphic as quasi-ordered fields.
\end{abstract}
\section{Introduction}
The theory of convex valuations and coarsenings of valuations is a special chapter
in classical valuation theory. It is a basic tool in algebraic and real algebraic geometry.
Surveys can be found in \salmacite{[LAM2]}, \salmacite{[LAN]} and
\salmacite{[PC]}. This special chapter is in turn closely related to ordered algebraic structures, see \cite{[Fu]}. In particular, an important isomorphism invariant of an ordered or valued
field is its rank as a valued field, which has several equivalent characterizations: via the ideals of the valuation ring, the value group, or the residue field, see \cite{[ZS]}.
\mn
This can be extended to ordered and valued fields with extra structure, giving a characterization completely analogous to the above, but taking into account the corresponding induced structure on the ideals, value group, or residue field. For example, in
\cite[Chapter 3]{[K]} the notion of the exponential rank of an ordered exponential field is introduced and analysed in light of the above classical tools. The exponential  rank measures the growth rate of the given exponential function, and is thus closely related to asymptotic analysis in the sense of G. H. Hardy \cite{[H]}.
\mn
In this paper, we push this analogy to the case of an ordered or valued
difference field. We work with quasi-ordered fields, see \cite{[F]}. In Section \ref{sectprelconv} we review classical notions and results on ordered or valued fields. We thereby present a uniform approach via quasi-orders, treating simultaneously the cases of ordered and valued fields.  Theorem \ref{wconv} gives a characterization of the rank of a quasi-ordered field in terms of coarsenings of its natural valuation. Descending down to the value group of the quasi-ordered field, and yet further down to the value set $\Gamma$ of the value group,  the rank and principal rank are finally characterized by the chain $\Gamma$, see  Theorems \ref{theorem1OSMT} and \ref{theorem2OSMT}. In Section \ref{partIIOSMT} we
start by a key remark regarding equivalence relations defined by
monotone maps on chains.  We describe in Theorem \ref{theorem3OSMT}  the rank of a quasi-ordered field via the equivalence relations induced by addition and
multiplication on the field. This approach allows us to develop in
Section \ref{diffanal} the notion of difference compatible
valuations and introduce the difference rank. We characterize in Theorem \ref{wconvsigma} the difference rank, in analogy to Theorem  \ref{wconv}. \cite[Lemma 1]{[S]} is a special case of our Corollary \ref{weakiso} on weak isometries. Corollary \ref{intersection} describes the set of fixed points of an automorphism $\sigma$ in terms of its difference rank, whereas Corollary \ref{omega} examines the special case of $\omega$-increasing or $\omega$-contracting automorphisms.
In the last Section
\ref{principalsigmarank} we describe the principal difference rank, see Theorem \ref{theorem3OSMTsigma} and its
Corollaries \ref{theorem1OSMTsigma}, \ref{theorem2OSMTsigma} and
\ref{arbitrary}. In Theorem \ref{last} we construct large families of quasi-ordered difference fields with distinct difference ranks.
\mn
Some closing comments are in place. The theory of well-quasi orders \cite{[Kru]} is currently a highly developed part of combinatorics with surprising applications in logic, mathematics and computer science. Quasi-ordered algebraic structures are interesting in their own right, and we will continue our investigations of these fascinating objects.  Quasi-orders  \cite{[Bir]} appear in the literature also as {\it preorders}, see e.g. \cite[p.1]{[Fu]}. However we will not use this terminology, in order to avoid confusion with the notion of preorders appearing in real algebraic geometry (e.g. in \cite{[Kr]}). The theory of quasi-ordered abelian groups is closely related to that of C-groups \cite{[H-M]} and has already found interesting applications in \cite{[L]} to the study of the asymptotic couple associated to a valued differential field.
Throughout the paper, Hahn groups and Hahn fields play a fundamental role. The group of automorphisms of Hahn structures have been extensively studied, see  \cite{[B]}, \cite{[DG]},  \cite{[Ho]} and  \cite{[S]}. In future work, we will analyse the behaviour of the difference rank as function defined on these automorphism groups.
\section{ The rank of a quasi-ordered field} \label{sectprelconv}

\mn A {\bf quasi-order  (q.o.)} on a set $S$ is a binary relation $\preceq$ which is reflexive and transitive.
Throughout this paper, we will deal only with {\bf total quasi-order}, i.e.  either $a\preceq b$ or $b\preceq a$, for any $a,\;b\in S$.
We will omit henceforth `total'. Note that an order is a q.o which is in addition anti-symmetric. In the latter case, we say that $S$ is an ordered set or a {\bf chain}.
The {\bf  induced  equivalence relation} is defined by  $a\asymp b$ if and only if ($a\preceq b$ and $b\preceq a$). We shall write $a \prec b $ if  $a\preceq b$ but $b\asymp a$ fails.  Note that $\preceq$ induces canonically a total order on $S/\asymp$. Conversely if $\asymp$ is an equivalence relation on a set $S$  such that  $S/\asymp$ is a total order,  then $\asymp$ induces canonically a q.o. on $S$.
A subset $E$ of $S$ is {\bf $\preceq$-convex} if for all $a, b, c$ in $S$, if $a\preceq c\preceq b$ and $a,\;b\in E$, then $c\in E$. We shall write convex instead of  $\preceq$-convex if the context is clear.
\mn  A {\bf quasi-ordered field} $(K,\preceq)$ is a field $K$ endowed with a quasi-order $\preceq$ which satisfies the  following compatibility conditions, for any $a, b, c \in K$.
\begin{description}
\item[qo1] If $a\asymp 0$,  then  $a=0$.
\item[qo2] If $0\preceq c$ and $a\preceq b$, then  $ac\preceq bc$.
\item[qo3] If $a\preceq b$ and $b\not\asymp c$ , then $a+c\preceq b+c$.
\end{description}
From {\bf qo2} one deduces that if $a\preceq b$ and $0\preceq c\preceq d$, then $ac\preceq bc\preceq bd$, so $ac\preceq bd$.\mn
 Given a valuation $w$ on $K$  we denote the {\bf valuation ring} by $K_w\,$, its {\bf group of units}  $K_w^{\times}$ by  $\cal{U}_w$,
its  {\bf valuation ideal} (i.e. its unique maximal ideal) by $I_w\,$, its {\bf value group} by $w(K^{\times})$ and {\bf residue field} $K_w/I_w$ by $Kw \>.$
\sn
An ordered field $(K, \leq)$ is a q.o. field. The valuation on a valued field $(K,w)$  induces a quasi-order: $a\preceq_{w} b$ if
and only if $w(b)\leq w(a)$, i.e. if and only if $ab^{-1}\in K_w$.
S. Fakhruddin \salmacite{[F]}  showed that if  $\preceq$  is  a q.o. on a field  $K$, then $\preceq$   is either an order or there is a (unique up to equivalence of valuations) valuation $v$  on $K$ such that
$\preceq\> = \preceq_{v}$. The dichotomy is achieved by considering the equivalence class $E_1$  of $1$ with respect to $\asymp \>$.  In the order case,  $E_1 = \{1\}$ and $\asymp \>$ is just equality.  The quasi-order is said to be a  {\bf proper quasi-order (p.q.o.)} if  $E_1 \not= \{1\}$. In this case, $E_1 \not= \{1\}$  is a non-trivial subgroup of $K^{\times}$ and $K^{\times}/E_1$ is an ordered abelian group.  Then  $\cal{U}_v$ is just $E_1$ and $v(K^{\times})\>$  is  $K^{\times}/E_1$.
In the p.q.o case  $a\succeq  0$ for all $a\in K\>$.
 \mn
Given two valuations $v$ and $w$ on $K$, recall that $w$  is said to be a \slind{coarsening} of $v$ ($w$ is coarser than $v$) or that $v$ a \slind{refinement} of $w$ ($v$ is finer than $w$) if  $K_v \subseteq K_w$. In case the inclusion of the valuation rings is strict, we add the predicate strict in the terminology coarser and finer. Note that $w$ is coarser than $v$ if and only if
$a\preceq_{v} b$ implies $a\preceq_{w} b\>$. If $\sim _1$ and $\sim _2$ are two equivalence relations
defined on the same set,  then $\sim _1$ is said to be {\bf coarser}
than $\sim _2$ (or $\sim_{2}$ {\bf finer} than $\sim_{1}$) if $\sim
_2$-equivalence implies $\sim _1$-equivalence.

\sn
Let us now fix a q.o. $\preceq$ on $K$ .  A valuation $w$ on $K$ is called \slind{convex} with respect to $\preceq$  if its valuation ring $K_w$  is convex.
It is called \slind{compatible} with $\preceq$  (or  $\preceq$ is compatible with $w$ or $w$ and  $\preceq$  are compatible)  if
 for all  $a,b\in K\>:$ $$ 0\preceq b \preceq a \;\;\;\Longrightarrow\;\;\;w(a) \leq w(b)\>.$$ Equivalently,  $w$ is compatible with $\preceq$
 if and only if for all $a,b\in K\>:$  $$0\preceq b \preceq a \;\;\;\Longrightarrow\;\;\; b \preceq_w a\>.$$
\begin{remark} \label{OSRGA1}
\sn

(i) If  $\preceq$  is an order, then this is the usual notion of compatibility for orders and valuations, see e.g. \salmacite{[LAM1]}, \salmacite{[LAM2]},  \salmacite{[PC]}, or
\salmacite{[PR1]}.
\sn

(ii) If  $\preceq =  \preceq _v$ is a p.q.o. then $w$ compatible with  $\preceq_v$  just means that for all $a,b\in K\>, v(a) \leq v(b) \;\;\;\Longrightarrow\;\;\;w(a) \leq w(b)\>.$
 This in turn just means that $K_v \subseteq K_w$  or $w$ is a coarsening of $v$, equivalently $\asymp_{w}$ is coarser than $\asymp_{v}$.
\end{remark}
\mn
The following gives the characterization of valuations  compatible with a quasi-order.  Theorem \ref{wconv} is  in complete analogy to the characterization of valuations  compatible with an order.  So for $\preceq$ an order, we omit the proof and  refer the reader to
\cite [Proposition~5.1]{ [LAM1]} , or \cite[Theorem~2.3 and Proposition~2.9]{ [LAM2]} , or \cite[Lemma~3.2.1]{ [PC]}, or \cite[ Lemma~7.2]{ [PR1]}  or \cite[ Proposition 2.2.4]{ [EP]}.
\begin{theorem}                                        \label{wconv}
Let $(K, \preceq)$  be a q.o. field and $w$ a valuation on $K$. The following assertions are equivalent:\sn 1)\ \ $w$ is
compatible with $\preceq$,\sn 2)\ \ $w$ is convex,
\sn 3)\ \ $I_w$ is convex,\sn
4)\ \ $I_w\prec 1\,$,
\sn 5)\ \  the quasi-order  $\preceq$ induces canonically via the residue map
$a \mapsto aw$ a quasi-order on the residue field $Kw\>.$
\end{theorem}
\begin{proof}
Assume $\preceq =  \preceq _v$ is a p.q.o.
Compatible valuations are clearly convex, this follows from the definitions. Conversely if $w$ is convex and $0 =v(1) \leq v(a)\>,$ i.e. $a \preceq 1\>,$ then $a\in K_w$ by convexity. So $w$ is a coarsening of $v$. This establishes the equivalence of 1) and 2).\sn
If $w$ is convex,  $a \preceq b$ with $b \in I_w\>,$ then $0 < w(b) \leq w(a)$ by compatibility, so $a\in I_w$. Conversely assume $I_w$  convex, and let $a \preceq b$ with $b \in K_w\setminus I_w$. If $a\notin K_w$ then $a^{-1} \in I_w$. Now $b^{-1} \preceq a^{-1}$, so
$b^{-1} \in I_w\>,$ a contradiction.  This establishes the equivalence of 2) and 3). \sn
If $I_w$ is convex, then $w$ is a coarsening of $v$, so $I_w \subseteq I_v\prec 1$. Conversely, assume $I_w \prec 1$ and let  $a \preceq b$ with $b \in K_w\>$. If $a\notin K_w\>,$ then $a^{-1} \in I_w\>.$ So  $a^{-1}b \in I_w$ whence  $a^{-1}b\prec 1$. Multiplying by $a$ gives $b \prec a$, a contradiction. This establishes the equivalence of 3) and 4). \sn
Now let $w$ be a coarsening of $v\>$. Then $v$ induces canonically a valuation $v/w$ on the residue field $Kw$, defined by  $v/w(aw) := \infty$ if $aw = 0$ and  $v/w(aw):=v(a)$ otherwise
(\salmacite{ [EP]} p. 44) . The p.q.o.  $\preceq_{v/w}$ is precisely the  induced well defined quasi-order in 5), i.e. $aw \preceq _{v/w} bw$ if and only if $a  \preceq _v b$ holds.  Conversely, let  $\preceq _{v/w}$
 be a p.q.o. on $Kw$ induced by the residue map. This means that  $aw \preceq _{v/w} bw$ if and only if $a  \preceq _v b$ holds.  Then $w$ is a coarsening of $v$  (see  \cite[ p. 45]{ [EP]}).  This establishes the equivalence of 1) and 5).
\end{proof}
\begin{remark}  \label{OSRGA2}
 If $\preceq$ is an order then the induced quasi-order in 5) is also an order, if $\preceq$ is a p.q.o then the induced quasi-order in 5) is also a p.q.o.
\end{remark}
Let  $(K,\preceq)$  be a q.o. field. We define its {\bf natural valuation}, denoted by $v$,  to be the finest $\preceq$- convex valuation of $K$. If  $(K, \leq)$  is ordered,  then the natural valuation  is the valuation $v$ whose
valuation ring $K_v$ is the convex hull of $\Q$ in $K$.  In this case, the
natural valuation on $K$ satisfies  $v(x+y) = \min\{v(x), v(y)\}$ if sign($x$) = sign($y$) and for all $a,b\in K\,: a\geq b>0\;\;\;\Longrightarrow\;\;\;v(a)\leq v(b)\>.$ It
is characterized by the fact that the induced order on  its residue field $Kv$  is archimedean, i.e. the only equivalence classes for the archimean equivalence relation (see definition below following Lemma \ref{wconvvg}) are those of 0 and 1. If $w$ is a coarsening of a convex valuation,
 then $w$ also is convex. Conversely, a convex subring
containing $1$ is a valuation ring, see \cite[Section 2.2.2]{[EP]}.
The set ${\cal R}$ of all valuation rings $K_w$ of convex valuations
$w\ne v$ (i. e. all strict corsenings of $v$) is totally ordered by
inclusion. Its order type is called the {\bf rank of the ordered field}
$K$. For convenience, we will identify it with
${\cal R}$. For example, the rank of an archimedean ordered field is
empty since its natural valuation is trivial (i.e. its valuation ring is the field itself). The rank of the
rational function field $K=\R(t)$ with any order is a singleton:
${\cal R}= \{K\}$. Theorem \ref{wconv} is a characterization of the elements of the rank of the ordered field $(K, \leq)$.  Note that
the rank of $(K, \leq)$  is  invariant under isomorphisms of ordered fields.
\sn
If  $(K,\preceq)$ is p.q.o. then the unique valuation $v$ such that $\preceq= \preceq _v$  is the natural valuation. A compatible valuation $w$ is a coarsening of $v$. We define the \slind{rank of  the valued field} $(K, v)$  to be the (order type of the) totally ordered set ${\cal R}$ of all strict corsenings of $v$. Thus, Theorem \ref{wconv} is a characterization of the elements of the rank of $(K,v)$.  Note that
the rank of $(K, v)$  is  invariant under isomorphisms of valued fields.
As we recalled in the proof of Theorem \ref{wconv}, the natural valuation $v$ induces canonically a valuation $v/w$ on the residue field $Kw$ and $v$ is the \slind{compositum} of $w$ and $v/w$
(see  \cite[ pp. 44-45]{ [EP]}) . The p.q.o.  $\preceq_{v/w}$ is precisely the  induced quasi-order in Theorem \ref{wconv} 5). If $w = v\>,$  then $v/w$  is trivial. Thus $v$
is characterized by the fact that the induced p.q.o on its residue field $Kv$  is \slind{trivial}, i.e. the only equivalence classes of $\asymp$ are those of $0$ and $1$.
\begin{remark} \label{prime}
The maximal ideals $I_w$ appearing in Theorem \ref{wconv} 4) are prime ideals of the valuation ring $K_v$.  The strict coarsenings $K_w$ of $K_v$ are the localizations of $K_v$ at the prime ideals $\{0\} \subseteq I \subset I_v$,  \cite[Lemma 2.3.1 p. 43]{[EP]}, \cite[Theorem 15, p. 40]{[ZS]}. Thus the rank is also isomorphic to the totally ordered (by reverse inclusion) set of prime ideals of $K_v$ which are strictly contained in the maximal ideal $I_v$.
\end{remark}
We now want to characterize the rank by going down to the value group. Let  $v$ be the natural valuation on the q.o. field  $(K,\preceq)$.
We set $G = v(K^{\times})$.
 Recall that the set of all convex subgroups
$G_w \ne \{0\}$ of the value group $G$ is totally ordered by
inclusion. Its order type is called the {\bf rank} of~$G$, it is an isomorphism invariant, see  \cite {[Fu]} or \cite{ [PC]}. To every convex valuation ring $K_w$, we
associate a convex subgroup $G_w\>:=\>\{v(a)\mid a\in K\,\wedge\
w(a)=0\}\>=\> v({\cal U}_w)\;.$  We call $G_w$ the {\bf convex
subgroup associated to $w$}. Note that $G_v=\{0\}$. Conversely,
given a convex subgroup $G_w$ of $v(K^{\times})$, we define $w: K \ra
v(K^{\times})/G_w$ by $w(a)=v(a)+G_w$. Then $w$ is a convex valuation with
$v({\cal U}_w)=G_w$ (and $v$ is strictly finer than $w$ if and only if
$G_w \not= \{0\}$). We call $w$ the {\bf convex valuation associated
to $G_w$}.
 We summarize the above discussion in the following lemma, for more details see \cite{[EP]}, or \cite {[Fu]} or \cite{ [PC]}.
\begin{lemma}\label{wconvvg}
The correspondence $K_w \mapsto G_w$ is an order preserving
bijection, thus ${\cal R}$ is (isomorphic to) the rank of $G$.
\end{lemma}
We now want to characterize the rank by going further down to the value set of the value group.
Recall that on the negative cone $G^{< 0}$ of an ordered abelian group $G$,
the {\bf archimedean equivalence} relation $\sim$ is defined by: $a\sim b$ if and only if
there is $n\in \N$ such that $a\geq nb$ and $b\geq na$.
Let $v_G$ be the map $a\mapsto
[a]_{\sim}\,$, where $[a]_{\sim}$ denotes the equivalence class of $a$. The order on $\Gamma:=G^{<0}/\sim$ is
the one induced by the order of $G^{<0}$. We call
$v_G(G^{<0}):=\Gamma$ the {\bf value set of $G$}. By convention we
also write $v_G(G):=\Gamma\cup\{\infty\}$ extending the archimedean equivalence
relation to the positive cone of $G$ by setting $v_G(g) := v_G(-g)$ and $v_{G}(0)=\infty>\Gamma$.
The map $v_G$ on $G$ satisfies the ultrametric
triangle inequality, and in particular we have: $v_G(x+y)
 = \min\{v_G(x), v_G(y)\}$ if sign($x$) = sign($y$). We call   $v_G$ the {\bf natural valuation on $G$}.

We now recall the relation between the rank of $G$ and the
value set $\Gamma$ of $G$. To $G_w \ne \{0\}$ a convex subgroup, we
associate $\Gamma_w := v_G(G_w ^{<0})$ a non-empty final segment of
$\Gamma$. Conversely, if $\Gamma _w$ is a non-empty final segment of
$\Gamma$, then $G_w = \{g\mid g\in G, v_G(g) \in \Gamma _w \} \cup
\{0\}$ is a convex subgroup, with $\Gamma _w = v_G(G_w)$. Let us
denote by $\Gamma ^{\rm fs}$ the set of non-empty final segments of
$\Gamma $, totally ordered by inclusion.  We summarize the above discussion in the following lemma, for more details see \cite{[EP]}, or \cite {[Fu]} or \cite{ [PC]}.
\begin{lemma}  \label{wconvvs}                            
The correspondence $G_w \mapsto \Gamma _w$ is an order preserving
bijection, thus the rank of $G$ is (isomorphic to) $\Gamma ^{\rm
fs}$.
\end{lemma}
Combining Lemmas \ref{wconvvg} and  \ref{wconvvs} we obtain the following result. Note that Theorem \ref{theorem1OSMT} will also follow, by a different argument, from Theorem \ref {theorem3OSMT} in the next section.
\begin{theorem} \label{theorem1OSMT}
The correspondence $K_w \mapsto \Gamma _w$ is an order preserving
bijection, thus ${\cal R}$ is (isomorphic to) $\Gamma ^{\rm fs}$.
\end{theorem} \mn
A final segment which has a least element is a \slind{principal
final segment}. It is of the form $\{\gamma'\mid\gamma'\in\Gamma,\gamma'\geq\gamma\}$, for some $\gamma \in \Gamma$.  Let $\Gamma ^*$  denote the set $\Gamma$ with its
reversed ordering. The proof of the following Lemma is now routine.
\begin{lemma}                               \label{vsr}
The map from $\Gamma$ to $\Gamma ^{\rm fs}$ defined by
$\gamma\mapsto\{\gamma'\mid\gamma'\in\Gamma,\gamma'\geq\gamma\}$ is
an order reversing embedding. Its image is the set of principal
final segments.  Thus $\Gamma ^*$ is (isomorphic to) the totally
ordered set of principal final segments.
\end{lemma}
\mn For the notions and results in this last paragraph of the section, we refer the reader to \cite{[Fu]} or \cite{[PC]} for more details. Recall that a convex subgroup $G_w$ of $G$ is called {\bf
principal generated by $g$}, $g\in G$, if $G_w$ is the minimal
convex subgroup containing $g$. The {\bf principal rank} of $G$ is
the subset of the rank of $G$ consisting of all principal $G_w \ne
\{0\}$ .
\begin{lemma}                               \label{cpcs}
Let $G_w \ne \{0\}$ be a convex subgroup. Then $G_w$ is principal if
and only if $v_G(G_w) = \Gamma _w$ is a principal final segment.
\end{lemma}
\begin{lemma}                               \label{star}
The map $G_w \mapsto \min \Gamma_w$  is an order reversing
bijection from the principal rank of $G$ onto $\Gamma$ . Thus the principal rank of $G$ is (isomorphic to) $\Gamma
^*$.
\end{lemma}
\sn
We set: $\mbox{\bf P}_K:=K^{\succeq 0}\setminus K_v$, where $K^{\succeq 0}: = \{a\in K \>; a\succeq  0\}$. A $K_w \in \cal{R}$ is \slind{principal
generated by $a$} for $a \in \mbox{\bf P} _K$ if $K_w$ is the
smallest (convex) subring containing $a$. We observe:
\sn
\begin{lemma}\label{principal} Let $K_w\in\cal{R}$. Then, $K_{w}$ is principal generated by $a$  if and only if
$K_{w}=\{b\in K:\;\exists n\in \N_{0} \> s.t.\;\;b\preceq_{v}  a^n\}$.
\end{lemma}
\begin{proof}
It is enough to verify that $\{b\in K:\;\exists n\in \N_{0}\;\;b\preceq_{v}  a^n\}$ is a subring of $K$. Let $b_{1}\preceq_{v} a^{n_{1}}$ and $b_{2}\preceq_{v}  a^{n_{2}}$. Then $b_{1}b_{2}\preceq_{v}  a^{n_{1}+n_{2}}$ and $b_{1}+b_{2}\preceq_{v} a^{max\{n_{1},n_{2}\}}$. Clearly, this ring contains $K_v$ and $a$.
\end{proof}
\sn
The {\bf principal rank} of
$K$ is the subset \gloss{${\cal R}^{\rm pr}$} of ${\cal R}$
consisting of all principal $K_w\in {\cal R}$. Combining the last three lemmas we obtain:
\begin{theorem} \label{theorem2OSMT}
The correspondence $K_w \mapsto \Gamma _w$ is an order preserving
bijection between  ${\cal R}^{\rm pr}$ and the principal rank of
$G$, thus ${\cal R}^{\rm pr}$ is (isomorphic to) $\Gamma ^*$.
\end{theorem}
Note that Theorem \ref{theorem2OSMT} will also follow, by a different argument, from Theorem \ref {theorem3OSMT} in the next section.
\begin{remark}\label{completion} It is straightforward to verify that
an order preserving isomorphism $\psi:\Gamma_{1}\rightarrow
\Gamma_{2}$ induces an order preserving isomorphism $\psi^{\rm
fs}:\Gamma_{1}^{\rm fs}\rightarrow\Gamma_{2}^{\rm fs}$
(\cite[p.19]{[R]}). Thus $\Gamma$ determines $\Gamma^{\rm fs}$ up to
isomorphism. It follows from Theorems \ref{theorem1OSMT} and
\ref{theorem2OSMT} that if two q.o. fields have isomorphic
principal ranks, then they have isomorphic ranks. In the next
section we shall hence focus our attention on the principal
rank.
\end{remark}

\section{The  principal rank via equivalence relations} \label{partIIOSMT}
\sn
We begin by the following key observation: \mn

\begin{remark}                              \label{mm}

Let $\varphi$ be a map from a q.o. ordered set $(S, \preceq)$ into itself,
and assume that $\varphi$ is q.o. preserving,\ i.\ e. \ $a\preceq a'$
implies $\varphi(a)\preceq \varphi(a')$, for all $a,\> a' \in S$.
Assume that $\varphi$ has an orientation or is {\bf oriented}, \ i.\ e. \
$\varphi(a)\succeq a$ for all $a\in S\>$ ($\varphi$ is a \slind{right shift})
or $\varphi(a)\preceq a$ for all $a\in S\>$ ($\varphi$ is a \slind{left
shift}). We set $\varphi^0(a):=a$ and
$\varphi^{n+1}(a):=\varphi(\varphi^n(a))$ for $n\in \N _{0}:=\N \cup
\{0\}$.
 It is then straightforward that the following defines an
equivalence relation on $S$: \sn
 (i) If $\varphi$ is a right shift,
set $a\sim_{\varphi} a'$ if and only if there is some $n\in\N_{0}$
such that $\varphi^n (a)\succeq a' \> \mbox{ and } \> \varphi^n
(a')\succeq a$ (equivalently for some $n,\;m\in \N_{0}$, $\varphi^n (a)\succeq a' \> \mbox{ and } \> \varphi^m(a')\succeq a\;$),
 \sn (ii) If $\varphi$ is a left shift, set
$a\sim_{\varphi} a'$ if and only if there is some $n\in\N_{0}$ such
that $\varphi^n (a)\preceq a' \>  \mbox{ and } \> \varphi^n
(a')\preceq a$ (equivalently for some $n,\;m\in \N_{0}$, $\varphi^n
(a)\preceq a' \> \mbox{ and } \> \varphi^m(a')\preceq a\;$).
\sn
(iii) The equivalence
classes $[a]_{\varphi}$ of $\sim_{\varphi}$ are $\preceq$-convex and
closed under application of $\varphi$. By the $\preceq$-convexity,
the quasi-order of $S$ induces an order on $S/{\sim}_{\varphi}$ such that $[a]_{\varphi}\prec [b]_{\varphi}$ if and only if $a'\prec b'$ for
all $a'\in [a]_{\varphi}$ and $b'\in [b]_{\varphi}\,$.
\sn
 Note that if $\varphi$ is the identity map $\mathbb{I}$,
then the equivalence relation $\sim_{\mathbb{I}}$  is just
$\asymp$  associated to the q.o., and is the finest one such that $S/\sim_{\mathbb{I}}$ is an ordered set.
\end{remark}
\mn
We exploit Remark \ref{mm} to give an
interpretation of the rank and principal rank as quotients via an
appropriate equivalence relation, thereby providing - as promised in the previous section-  alternative
proofs for Theorem \ref{theorem1OSMT} and Theorem
\ref{theorem2OSMT}. It is precisely this approach that we will
generalize to the difference rank in Section \ref{principalsigmarank}.
Let  $v$ be the natural valuation on the q.o. field  $(K,\preceq)$. Recall that $\mbox{\bf P}_K$ denotes $K^{\succeq 0}\setminus K_v$.
Consider the following commutative diagram: \n
\parbox[c]{.4\textwidth}{
\begin{center}
\setlength{\unitlength}{0.002\textwidth}
\begin{picture}(200,250)(0,20)
\put(50,250){\bbox{${\bf P} _K$}} \put(50,150){\bbox{$G ^{<0}$}}
\put(50,50){\bbox{$v_G (G)$}} \put(150,250){\bbox{${\bf P} _K$}}
\put(150,150){\bbox{$G ^{<0}$}} \put(150,50){\bbox{$v_G (G)$}}
\put(80,50){\vector(1,0){40}} \put(80,150){\vector(1,0){40}}
\put(80,250){\vector(1,0){40}} \put(50,230){\vector(0,-1){60}}
\put(150,230){\vector(0,-1){60}} \put(50,130){\vector(0,-1){60}}
\put(150,130){\vector(0,-1){60}} \put(60,200){\bbox{$v$}}
\put(160,200){\bbox{$v$}} \put(60,100){\bbox{$v_G$}}
\put(160,100){\bbox{$v_G$}} \put(100,260){\bbox{$\varphi$}}
\put(100,160){\bbox{$\varphi_{G}$}}
\put(100,60){\bbox{$\varphi_{\Gamma}$}} \put(100,200){\bbox{\tiny\rm
///}} \put(100,100){\bbox{\tiny\rm ///}}
\end{picture}
\end{center}}\hfill\parbox[c]{.58\textwidth}{with
 $\varphi(a)\>:=\> a^2$ for all $a\in
{\bf P}_K\>,$ $\varphi$ is a right shift, \bn\bn $\varphi_{G}(v(a))\>:=\>v(\varphi(a))$ for all
$a\in {\bf P}_K\>,$\bn that is $\varphi_{G}(g)= 2g$ for all $g\in
G^{<0}\>,$  $\varphi_{G}$  is a left shift and \bn\bn $\varphi_{\Gamma}(v_G(g))\>:=\> v_G
(\varphi_{G} (g))$ for all $g\in G^{<0}\>,$ \bn that is
$\varphi_{\Gamma}(\gamma)=\gamma$ for all $\gamma\in\Gamma\>$, so
that $\varphi_{\Gamma}$ is just the identity map.} \bn By
Remark \ref{mm}, we can work with
the equivalence relations associated to the following oriented maps: the q.o. preserving map
$\varphi$ and the order preserving maps $\varphi_G$ and $\varphi_{\Gamma}$ (as defined on the right hand side of the above diagram). Note that
$\>\sim_{\varphi_G}\>$ is just archimedean equivalence on $G$ and
$\sim_{\varphi_{\Gamma}}$ is just equality on $\Gamma$. The following straightforward observation
will be useful for the proof of Theorem \ref{theorem3OSMT} below:
\begin{lemma}                               \label{corsavarphi}
The equivalence classes of $\sim_{\varphi}$ are closed under multiplication.
\end{lemma}
\begin{proof}
 The proof is similar to that of Lemma \ref{principal}. Let $a,\;b \in {\bf  P}_K$, and without loss of generality assume that $a\preceq b$ and $a\sim_{\varphi} b$.
We show that $ab\sim_{\varphi} a$.
Let $n\in \N_{0}$,  such that $b\preceq a^{2^n}$. By axiom qo$2$, $ab\preceq b^2$.
Thus $b^2\preceq a^{2^n}b$ and $ab\preceq a^{2^n}b$. So, $ab\preceq a^{2^{n+1}}$. Since $1\preceq b$, by axiom qo$2$, we get that $a\preceq ab$.
 Therefore, $ab\sim_{\varphi} a$.
\end{proof}
\begin{remark}\label{mmc} We note that
\begin{equation}
\varphi^n _G (v(a))\> = \> v(\varphi ^n (a)) \mbox{ and } \varphi^n _{\Gamma} (v_G(g))\> = \> v_G(\varphi^n _G (g))
\end{equation}
thus
\begin{equation}
a\sim_{\varphi} a' \mbox{ if and only if } v(a)\sim_{\varphi_{G}}v(a')
\mbox{ if and only if } v_G (v(a))\sim_{\varphi_{\Gamma}} v_G
(v(a'))\>
\end{equation}
Thus we have an order reversing bijection from ${\bf
P}_K/\sim_{\varphi}$ onto $\Gamma/\sim_{\varphi_{\Gamma}}=\Gamma$.
Thus the chain $[{\bf P}_K/\sim_{\varphi}]^{\rm is}$ of non-empty initial
segments of ${\bf P}_K/\sim_{\varphi}$ ordered by inclusion is
isomorphic to $\Gamma^{\rm fs}$. In particular,  initial segments which have a last element are in bijective correspondence to principal final segments.
Thus  the subchain of $[{\bf P}_K/\sim_{\varphi}]^{\rm
is}$ of initial segments which have a last element is isomorphic to $\Gamma ^*$ \footnote{Note that the subchain of $[{\bf P}_K/\sim_{\varphi}]^{\rm
is}$ of initial segments which have a last element is isomorphic to $[{\bf P}_K/\sim_{\varphi}]$ itself.}
Therefore, as promised in the previous section, Theorems \ref{theorem1OSMT} and \ref{theorem2OSMT} will now follow from the
following result:
\end{remark}
\begin{theorem}\label{theorem3OSMT}
The rank ${\cal R}$ is isomorphic to the chain $[{\bf
P}_K/\sim_{\varphi}]^{\rm is}$ and the principal rank ${\cal R}^{\rm
pr}$ is isomorphic to the subchain of $[{\bf P}_K/\sim_{\varphi}]^{\rm
is}$ of initial segments which have a last element.
\end{theorem}
\begin{proof}
First we note that if $K_w$ is a convex valuation ring, then clearly
$K_w^{\succ 0}\setminus K_v^{\succ 0}$ is an initial segment of ${\bf P}_K$.
Moreover by Lemma \ref{principal}  if $K_w$
is principal generated by $a$, then $[a]_{\sim _{\varphi}}$ is the last class. Furthermore, if $K_w$ intersects an
equivalence class $[a]_{\sim_{\varphi}}$ then it must contain it,
since the sequence $a^n; n\in \N_0$ is cofinal in
$[a]_{\sim_{\varphi}}$ and $K_w$ is a convex subring. We conclude
that $(K_w^{\succ 0}\setminus K_v^{\succ 0})/{\sim}_{\varphi}$ is an initial
segment of ${\bf P}_K/{\sim}_{\varphi}$.
Conversely set ${\cal I}_w\>=\> \{[a]_{\varphi}\mid a\in
K_w^{\succ 0}\setminus K_v^{\succ 0}\}\>.$ Given ${\cal I}\in [{\bf P}_K/
{\sim}_{\varphi}]^{\rm is}$, we show that there is a convex
valuation ring $K_w$ such that ${\cal I}_w = {\cal I}$. Given ${\cal
I}$, let $(\bigcup {\cal I})$ denote the set theoretic union of the
elements of ${\cal I}$ and $-(\bigcup {\cal I})$ the set of additive
inverses. Set $K_w=-\left(\bigcup {\cal I}\right)\cup K_v \cup
\left(\bigcup {\cal I}\right)\>.$ We claim that $K_w$ is the
required ring. Clearly, ${\cal I}_w = {\cal I}$. Further $K_w$ is
convex (by its construction), and strictly contains $K_v$. We leave
it to the reader, using Lemma \ref{corsavarphi} and  Lemma \ref{principal}, to verify that
$K_w$ is a ring, and that $K_w$ is principal generated by $a$ if
$[a]_{\sim {\varphi}}$ is the last element of ${\cal I}$.
\end{proof}
\section{The difference analogue of the rank}         \label{diffanal}
In this section, we develop a difference analogue of what has been
reviewed above. That is, we develop a theory of difference
compatible valuations, in analogy to the theory of convex
valuations. The automorphism will play the role that multiplication
plays in the previous case. \sn Let $(K, \preceq)$ be a q.o. field and $\sigma$ be a {\bf q.o. preserving}
field automorphism of $K$, that is, $a\preceq a'$
if and only if $\sigma(a)\preceq \sigma(a')$, for all $a,\> a' \in K$. We say that $(K, \preceq, \sigma)$  is a {\bf q.o. difference field.}
\begin{remark} \label{ordervsnatval}
Let $(K, \leq, \sigma)$  be an ordered difference field. Recall that the natural valuation $v$ on $K$ is defined by archimedean equivalence. Since archimedean equivalence is preserved under order preserving automorphisms, we see that  $\sigma$ is also  $\preceq_v$ preserving (so that $(K, \preceq_v, \sigma)$  is a q.o.  difference field). The converse fails: Consider the field of real Laurent series $K:=\R((t))$ endowed with the lexicographic order and the corresponding natural valuation $v_{\min}$ (see definitions following Corollary \ref{omega} below).  The map $t\mapsto (-t)$ defines a field automorphisme $\sigma$ on $K$ which clearly preserves $v_{\min}$ but not the lexicographic order on $K$.
\end{remark}
\sn
Now let
$(K, \preceq, \sigma)$ be a non-trivial (i.e.  $\sigma\not=$ identity) q.o. difference field and $v$ its natural valuation. By definition,
 $\sigma$ satisfies for all $a,b\in K\>: v(a)\leq v(b)\;\mbox{
if and only if }v(\sigma(a))\leq v(\sigma(b))$ and thus induces an
order preserving automorphism $\sigma_G$ and $\sigma_{\Gamma}$ such
that the following diagram commutes: \n
\parbox[c]{.4\textwidth}{
\begin{center}
\setlength{\unitlength}{0.002\textwidth}
\begin{picture}(200,250)(0,20)
\put(50,250){\bbox{${\bf P}_K$}} \put(50,150){\bbox{$G ^{<0}$}}
\put(50,50){\bbox{$v_G (G)$}} \put(150,250){\bbox{${\bf P}_K$}}
\put(150,150){\bbox{$G ^{<0}$}} \put(150,50){\bbox{$v_G (G)$}}
\put(80,50){\vector(1,0){40}} \put(80,150){\vector(1,0){40}}
\put(80,250){\vector(1,0){40}} \put(50,230){\vector(0,-1){60}}
\put(150,230){\vector(0,-1){60}} \put(50,130){\vector(0,-1){60}}
\put(150,130){\vector(0,-1){60}} \put(60,200){\bbox{$v$}}
\put(160,200){\bbox{$v$}} \put(60,100){\bbox{$v_G$}}
\put(160,100){\bbox{$v_G$}} \put(100,260){\bbox{$\sigma$}}
\put(100,160){\bbox{$\sigma_{G}$}}
\put(100,60){\bbox{$\sigma_{\Gamma}$}} \put(100,200){\bbox{\tiny\rm
///}} \put(100,100){\bbox{\tiny\rm ///}}
\end{picture}
\end{center}}\hfill\parbox[c]{.58\textwidth}{with
  $\sigma_{G}(v(a))\>:=\>v(\sigma(a))$ for all $a\in
{\bf P}_K\>,$ \bn\bn and \bn\bn $\sigma_{\Gamma}(v_G(g))\>:=\> v_G
(\sigma_{G} (g))$ for all $g\in G^{<0}\>.$} \bn Now let $w$ be a
convex valuation on $K$. Say $w$ is {\bf $\sigma$-compatible} if for
all $a,b\in K\>: w(a)\leq w(b)\;\mbox{ if and only if
}w(\sigma(a))\leq w(\sigma(b))\>.$  Thus $w$ is  $\sigma$-compatible if and only if $\sigma$  preserves the q.o. $\preceq _w$.
\bn The subset ${\cal R
_{\sigma}}:=\{\> K_w \in {\cal R}\>;\> w \mbox{ is } \sigma$-\mbox{
compatible }\} is the {\bf $\sigma$-rank} of $(K,\preceq,\sigma)$.
Similarly, the subset of all convex subgroups $G_w \ne \{0\}$ such
that $\sigma_G(G_w)=G_w$, i.e $G_w$ is {\bf $\sigma_G$- invariant}, is the {\bf
$\sigma$-rank} of~$G$. Finally, we denote by
$\sigma_{\Gamma}$-$\Gamma ^{\rm fs}$ the subset of non-empty final segments
$\Gamma _w$ such that $\sigma_{\Gamma}(\Gamma_w)=\Gamma_w$, i.e. $\Gamma _w$  is {\bf $\sigma_{\Gamma}$- invariant}.
\mn
The
following Theorem \ref{wconvsigma}, Lemmas \ref{wconvvgsigma} and \ref{wconvvssigma} are analogues of Theorem \ref{wconv}, Lemma \ref{wconvvg} and
Lemma \ref{wconvvs} respectively. They are verified by straightforward computations, using
basic properties of valuations rings on the one hand and of
automorphisms on the other (e.g.  $\sigma(A\setminus B) =\sigma(A)\setminus \sigma(B)$,  $\sigma(A) \subseteq B$ if and only
if $A \subseteq \sigma^{-1}(B)$ and $\sigma(A) \subseteq B$ if and
only if $\sigma(-A) \subseteq -B$). The equivalence of 1) and 7) in Theorem \ref{wconvsigma} follows from the compatibility of $\sigma$  with $w$ on the one hand, and from the definition of the induced q.o. on $Kw$ on the other. We call $K_w$ $\sigma$-compatible if any of the
equivalent conditions below holds.
\begin{theorem}                                        \label{wconvsigma}
The following assertions are equivalent for a convex valuation
$w\>$:\n 1)\ \ $w$ is $\sigma$--compatible\n 2)\ \ $w$ is
$\sigma^{-1}$--compatible\n 3)\ \ $\sigma(K_w)\>=\>K_w$\n 4)\ \
$\sigma(I_w)\>=\>I_w$ \n
5)\ \
$\sigma({\cal U}_w)\>=\>{\cal U}_w$
\n 6)\ \ $\sigma(K_w^{\succ 0}\setminus
K_v^{\succ 0})\>=\>K_w^{\succ 0}\setminus K_v^{\succ 0}$ \n 7)\ \ the map
$\sigma w: Kw\rightarrow Kw$ defined by $aw\mapsto \sigma(a)w$ is well-defined
and is a q.o. (with respect to the
induced q.o. on $Kw$ ) preserving field automorphism of $Kw$ .
\end{theorem}
\begin{remark}\label{sixandsixprime}
Let $(K, \leq, \sigma)$  be an ordered field with natural valuation $v$. In this case,  condition 7)  on $\sigma w$ in Theorem  \ref{wconvsigma} is referring to the induced order on the residue field $Kw$. Consider instead the following condition:\n
8)\ \ the map
$\sigma w: Kw\rightarrow Kw$ defined by $aw\mapsto \sigma(a)w$ is well-defined
and is a q.o. (with respect to the
 q.o.  $\preceq _{v/w}$  on $Kw$ ) preserving field automorphism of $Kw$ .\sn
We observe that 7) implies 8). Indeed, $ \sigma w$ is assumed to be order preserving on $Kw$ by 7).  Now $(Kw)(v/w) = Kv$ (see \cite[ Lemma 2.1]{[K-K]}). Therefore $v/w$ has archimedean residue field and is thus the natural valuation on the ordered field $Kw$. By Remark
 \ref{ordervsnatval} we obtain the assertion.
\end{remark}
\begin{remark} \label{sigmaprime}
The maximal ideals $I_w$ appearing in Theorem \ref{wconvsigma} 4) are $\sigma$- invariant prime ideals  (also called transformally prime ideals in \cite{[C]}) of the valuation ring $K_v$ and the coarsenings $K_w$ are just the localizations of $K_v$ at those  $\sigma$- invariant prime ideals, see \cite[ Lemma 2.3.1 p. 43]{[EP]}. Thus the $\sigma$- rank is also characterized by the chain of $\sigma$- invariant  prime ideals of $K_v$.
\end{remark}
\begin{lemma}\label{wconvvgsigma}
The correspondence $K_w \mapsto G_w$ is an order preserving
bijection from  ${\cal R}_{\sigma}$ onto the $\sigma_G$-rank of $G$.
\end{lemma}
\begin{lemma}  \label{wconvvssigma}                            
The correspondence $G_w \mapsto \Gamma _w$ is an order preserving
bijection from the $\sigma_G$-rank of $G$ onto
$\sigma_{\Gamma}$-$\Gamma ^{\rm fs}$.
\end{lemma}
\sn We deduce from Lemma \ref{wconvvgsigma} and Lemma \ref{wconvvssigma} that the $\sigma$-rank is the order type of $\sigma_{\Gamma}$-$\Gamma ^{\rm fs}$:
\begin{theorem} \label{theorem1OSMTsigma}
The correspondence $K_w \mapsto \Gamma _w$ is an order preserving
bijection from  ${\cal R}_{\sigma}$ onto $\sigma_{\Gamma}$-$\Gamma ^{\rm fs}$.
\end{theorem} \mn
We now exploit this observation. An automorphism $\sigma$ is an {\bf
isometry} if $v(\sigma(a))=v(a)$ for all $a\in K$, equivalently
$\sigma_G$ is the identity automorphism, and a {\bf weak isometry}
if $\sigma_{\Gamma}$ is the identity automorphism. Every isometry is
a weak isometry. Note that if $\Gamma$ is a rigid chain (i.e the only order preserving automorphism is the identity map), then
$\sigma$ is necessarily a weak isometry. If $\sigma$ is a weak
isometry, then
$\sigma_{\Gamma}(v_G(g))\>=\>v_G(\sigma_{G}(g))\>=\>v_G(g)$, thus
$g$ is archimedean equivalent to $\sigma_{G}(g)$ for all $g$, and so
every convex subgroup is $\sigma_G$-invariant.
\begin{corollary} \label{weakiso}
 If $\sigma$ is a
weak isometry, then ${\cal R}_{\sigma}\>=\> {\cal R}$.
\end{corollary}
\begin{corollary}\label{intersection}
The correspondence $K_w \mapsto \min \Gamma_w$ is an order
(reversing) isomorphism from ${\cal R}_{\sigma}\cap {\cal R}^{\rm
pr}$ onto the chain $\{\gamma\>;\> \sigma_{\Gamma}(\gamma) =
\gamma\}$ of fixed points of $\sigma_{\Gamma}$.
\end{corollary}
\begin{proof}
By  Lemma \ref{cpcs},  set $\mbox{ min } \Gamma_w:= \gamma_0$. By Lemma \ref{wconvvgsigma} and  Lemma \ref{wconvvssigma},  $\Gamma_w$ in invariant under $\sigma_{\Gamma}$.
Since $\sigma_{\Gamma}$ is order preserving, we must have  $\sigma_{\Gamma}(\gamma_0) = \gamma_0$
\end{proof}
\mn At the other extreme $\sigma$ is said to be  {\bf $\omega$-increasing}  if $a^n \prec \sigma(a)$ for all $n\in\N_0$ and all $a\in {\bf P}_K$, and {\bf $\omega$-contracting} if  $\sigma^{-1}$ is $\omega$-increasing.
\begin{remark}\label{strict left shift}
Note that $\sigma$ is $\omega$-increasing (respectively, $\omega$-contracting)  if and only if
$\sigma_{\Gamma}$ is a {\bf strict left shift}, that is,  $\sigma_{\Gamma}(\gamma) < \gamma$ for all $\gamma \in \Gamma\>$
 (respectively, a {\bf strict right shift}, i.e. $\sigma_{\Gamma}(\gamma)> \gamma$ for all $\gamma \in \Gamma\>$).
Thus if $\sigma$ $\omega$-increasing or $\omega$-contracting, then $\sigma_{\Gamma}$ has no
fixed points.
\end{remark}
\begin{corollary}\label{omega}
If $\sigma$ is $\omega$-increasing or $\omega$-contracting, then ${\cal R}_{\sigma}\cap {\cal
R}^{\rm pr}$ is empty.
\end{corollary}
\mn Recall that the {\bf Hahn group} \cite{[Ha]} over the chain $\Gamma$ and
components $\R$, denoted ${\bf H}_{\Gamma} \R$, is the totally
ordered abelian group whose elements are formal sums $\>g:= \sum
g_{\gamma} 1_{\gamma}\>$, with  well-ordered $\mbox{ support } g:
=\{\gamma\>;\> g_{\gamma}\not= 0\}\>.$ Here $g_{\gamma}\in \R$ and $1_{\gamma}$ denotes the characteristic function on the singleton $\{\gamma\}$.
 Addition is pointwise and the order lexicographic. Similarly, given a field $F$,
 the field of {\bf generalized power series} over the ordered abelian group $G$ (or {\bf Hahn field} over $G$) with coefficients in $F$,
 denoted $\mathbb{F}:=F((G))\>$, is the field whose elements are formal series $\>s:= \sum s_{g}
t^g\>$, with well-ordered $\mbox{ support } s: =\{g\>;\> s_{g}\not=
0\}\>.$ Addition is pointwise, multiplication is given by the usual
convolution formula. The field $\mathbb{F}$ has the same characteristic as that of $F$. The canonical valuation $v_{\min}$ on $\mathbb{F}$ is defined by $v_{\min}(s) : = \min \mbox{ support } s$ for $s \not= 0$. Its value group is $G$ and its residue field is $F$. Thus  ($\mathbb{F}, \preceq _{v_{\min}}$ ) is a q.o. field.
 If $F$ is an ordered field, its order extends to the lexicographic order on $\mathbb{F}$: a series $s$ is positive if and only if the coefficient of $t^{v_{\min}(s)}$ is positive in $F$.  Thus, in that case $(\mathbb{F}, \leq)$ is an ordered field.
Hahn fields are maximally valued: they admit no proper immediate extension, that is, no proper valued field extension preserving the value group and the residue field. They were extensively studied e.g. by Hahn \cite{[Ha]} and in the seminal paper of Kaplansky \cite{[KA]}.

\begin{lemma}\label{lifting}
Any order preserving automorphism $\sigma_{\Gamma}$ of the chain
$\Gamma$ lifts to an order preserving automorphism $\sigma_G$ of the
Hahn group $G$ over $\Gamma$, and $\sigma_G$ lifts in turn to a q.o.
preserving automorphism $\sigma$ of the Hahn field over $G$.
\end{lemma}
\begin{proof}
Set $\>\sigma_G(\sum g_{\gamma} 1_{\gamma}):= \sum g_{\gamma}
1_{\sigma_{\Gamma}(\gamma)}\>$. It is straightforward to verify that the thus defined $\>\sigma_G$ induces the given automorphism $\sigma_{\Gamma}$ on $\Gamma$. Thus $\>\sigma_G$ is a lifting of  $\sigma_{\Gamma}$. Now set $\>\sigma(\sum s_{g} t^g):= \sum s_{g} t^{\sigma_{G}(g)}\>.$ Again, it is clear that $\sigma$ induces $\>\sigma_G$ on $G$.
Thus $\sigma$ is a lifting of  $\>\sigma_G$ as asserted.
\end{proof}
\begin{corollary} \label{firstconstruction}
Given any order type $\tau$ there exists an
 ordered difference field $(K, \leq, \sigma)$,  and also a p.q.o. difference field $(K, \preceq, \sigma)$ such that the
order type of ${\cal R}_{\sigma}\cap {\cal R}^{\rm pr}$ is $\tau$.
\end{corollary}
\begin{proof}
Set $\mu: = \tau^*$, and consider e.g. the linear ordering
$\Gamma:=\sum _\mu \Q^{\geq 0}$, that is, the concatenation of $\mu$
copies of the non-negative rationals. Fix a non-trivial order
automorphism $\eta$ of $\Q^{>0}$. Define $\sigma _{\Gamma}$ to be
the uniquely defined order automorphism of $\Gamma$ fixing every
$0\in \Q^{\geq 0}$ in every copy and extending $\eta$  on
every copy.  It is clear that the order type of the chain of fixed points (the zeros in every copy) of $\sigma _{\Gamma}$ is  $\mu$.  Set e.g. $G:={\bf H}_{\Gamma} \R$. By Lemma \ref{lifting},
$\sigma_{\Gamma}$ lifts canonically to $\sigma_G$ on $G$. Now consider
e.g. the ordered field $\mathbb{F}:= \R((G))$. Again by Lemma \ref{lifting}, $\sigma_G$ lifts canonically to an order
automorphism $\sigma$ of $\mathbb{F}$. This is our required $\sigma$, by Corollary  \ref{intersection}. To obtain a p.q.o difference field, take $F$ any field and the corresponding $(\mathbb{F}, \preceq _{v_{\min}}, \sigma)$.
\end{proof}
\sn In the next section, we will exploit appropriate equivalence
relations to define the principal difference rank and construct
difference fields of arbitrary principal difference rank.
\section{The $\sigma$-rank and principal $\sigma$-rank via equivalence relations} \label{principalsigmarank}
Let $(K, \preceq, \sigma)$  be a q.o. difference field.  As promised in Section \ref{partIIOSMT}, we now exploit Remark \ref{mm} to give an
interpretation of the $\sigma$- rank and define the principal  $\sigma$-rank as quotients via
appropriate equivalence relations.
Our aim is to state and prove the analogues to Theorems
\ref{theorem3OSMT}, \ref{theorem1OSMT} and \ref{theorem2OSMT}.
We recall that the q.o. preserving maps considered in Remark \ref{mm} are assumed to be oriented.
Moreover, scrutinizing the proof of Theorem \ref{theorem3OSMT} we
quickly realize that we need Lemma \ref{corsasigma} below, an analogue of Lemma
\ref{corsavarphi}. Thus we need further assumptions on $\sigma$, to ensure that $\sigma$ satisfies Lemma  \ref{corsasigma}. For simplicity from now on we will assume
that  $\sigma$ or $\sigma^{-1}$ satisfy $\sigma(a) \succeq a^2$ for all $a\in {\bf P}_K$. Note that this implies that $\sigma(a) \succ a$, so $\sigma$ is an oriented  strict right-shift.  Note that
our condition on $\sigma$ is fulfilled for $\omega$-increasing or $\omega$-contracting
automorphisms. \mn A convex subring $K_w \ne K_v$ is
\slind{$\sigma$-principal generated by $a$} for $a \in \mbox{\bf P}
_K$ if $K_w$ is the smallest convex $\sigma$-compatible subring
containing $a$. The {\bf $\sigma$-principal rank} of $K$ is the
subset \gloss{${\cal R}_{\sigma}^{\rm pr}$} of ${\cal R}_{\sigma}$
consisting of all $\sigma$-principal $K_w\in {\cal R}$.  We will use the analogue of Remark \ref{mmc}:
\begin{remark} \label{mmcsigma}
The maps
$\sigma$, $\sigma_G$ and $\sigma_{\Gamma}$ are q.o. preserving and
we can define the corresponding equivalence relations
$\sim_{\sigma}$, $\sim_{\sigma_G}$ and $\sim_{\sigma_{\Gamma}}$. As
before we have
\begin{equation} \label{reduction}
a\sim_{\sigma} a' \mbox{ if and only if } v(a)\sim_{\sigma_{G}}v(a')
\mbox{ if and only if } v_G (v(a))\sim_{\sigma_{\Gamma}} v_G
(v(a'))
\end{equation}
Thus we have an order reversing bijection from ${\bf
P}_K/\sim_{\sigma}$ onto $\Gamma/\sim_{\sigma_{\Gamma}}$. Thus the
chain $[{\bf P}_K/\sim_{\sigma}]^{\rm is}$ of initial segments of
${\bf P}_K/\sim_{\sigma}$ ordered by inclusion is isomorphic to
$(\Gamma/\sim_{\sigma_{\Gamma}}) ^{\rm fs}$. As before, the subchain of initial segments which have a last element is isomorphic to $(\Gamma/\sim_{\sigma_{\Gamma}}) ^*$.
\end{remark}
\begin{lemma}                               \label{corsasigma}
The equivalence classes of
$\sim_{\sigma}$ are closed under $\sigma$ and under multiplication.
\end{lemma}
\begin{proof}
The condition on $\sigma$ implies by
induction that $\sigma^n(a) \succeq a^{2^n}$. Thus given $n\in\N_0\>,$
there exists $l\in \N_0$ such that $\sigma^l(a) \succeq a^n$. Thus $a\sim_{\sigma}\sigma (a)$. So the
 equivalence classes of $\sigma$ are closed under  $\sigma$ .
 Recall that the natural valuation $v_G$ on $G$ satisfies $v_G(x+y)
 = \min\{v_G(x), v_G(y)\}$ if sign($x$) = sign($y$).
Again one easily deduces from this fact and the equivalences (\ref{reduction})  above that the
 equivalence classes of $\sigma$ are closed under multiplication. Indeed assume that  $a\sim_{\sigma} b$ and  $a\sim_{\sigma} c$. We want to show that  $a\sim_{\sigma} bc$. Set $x: = v(b)$, $y: = v(c)$ and $z: = v(a) \in G ^{<0}\>$.
By the first equivalence in (\ref{reduction}), it is enough to show that  $v(a)\sim_{\sigma_{G}}v(bc)$ i.e. that $x + y \sim_{\sigma_{G}} z$. By the second equivalence  in (\ref{reduction}),  it is enough to show that $ v_G (x + y)\sim_{\sigma_{\Gamma}} v_G (z)\>$. Without loss of generality $ v_G (x + y) = v_G(x)$. But since $a\sim_{\sigma} b$ it follows by (\ref{reduction}) that  $ v_G (x)\sim_{\sigma_{\Gamma}} v_G (z)$ as required.
\end{proof}
We can now prove the analogue of Theorem \ref{theorem3OSMT}:
\begin{theorem}\label{theorem3OSMTsigma}
The $\sigma$-rank ${\cal R_{\sigma}}$ is isomorphic to $[{\bf
P}_K/\sim_{\sigma}]^{\rm is}$ and the principal $\sigma$-rank ${\cal
R}_{\sigma}^{\rm pr}$ is isomorphic to the subset of $[{\bf
P}_K/\sim_{\sigma}]^{\rm is}$ of initial segments which have a last
element.\footnote{Note that the subchain of $[{\bf P}_K/\sim_{\sigma}]^{\rm
is}$ of initial segments which have a last element is isomorphic to $[{\bf P}_K/\sim_{\sigma}]$ itself.}
\end{theorem}
\begin{proof}
First we note that if $K_w$ is a convex $\sigma$-compatible
valuation ring, then clearly $K_w^{\succ 0}\setminus K_v^{\succ 0}$ is an
initial segment of ${\bf P}_K$. Furthermore, if $K_w$ intersects a
$\sigma$- equivalence class $[a]_{\sim_{\sigma}}$ then it must
contain it, since the sequence $\sigma(a)^n; n\in \N_0$ is cofinal
in $[a]_{\sim_{\sigma}}$ and $K_w$ is a convex subring. We conclude
that $(K_w^{\succ 0}\setminus K_v^{\succ 0})/{\sim}_{\sigma}$ is an initial
segment of ${\bf P}_K/{\sim}_{\sigma}$ and moreover $[a]_{\sim
_{\sigma}}$ is the last class in case $K_w$ is $\sigma$- principal
generated by $a$. Conversely set ${\cal I}_w\>=\> \{[a]_{\sigma}\mid
a\in K_w^{\succ 0}\setminus K_v^{\succ 0}\}\>.$ Given ${\cal I}\in [{\bf P}_K/
{\sim}_{\sigma}]^{\rm is}$, we show that there is a
$\sigma$-compatible convex valuation ring $K_w$ such that ${\cal
I}_w = {\cal I}$. Given ${\cal I}$, let $(\bigcup {\cal I})$ denote
the set theoretic union of the elements of ${\cal I}$ and $-(\bigcup
{\cal I})$ the set of additive inverses. Set $K_w=-\left(\bigcup
{\cal I}\right)\cup K_v \cup \left(\bigcup {\cal I}\right)\>.$ We
claim that $K_w$ is the required ring. Clearly, ${\cal I}_w = {\cal
I}$. Further $K_w$ is convex (by its construction), and strictly
contains $K_v$. We leave it to the reader, using Lemma
\ref{corsasigma}, to verify that $K_w$ is a $\sigma$-compatible
subring, and that $K_w$ is $\sigma$-principal generated by $a$ if
$[a]_{\sim {\sigma}}$ is the last element of ${\cal I}$.
\end{proof}
We now deduce from this theorem combined with Remark \ref{mmcsigma} the promised analogues of Theorems  \ref{theorem1OSMT} and \ref{theorem2OSMT} respectively:
\begin{corollary} \label{theorem2OSMTsigma}
${\cal R}_{\sigma}$ is (isomorphic to) $(\Gamma/\sim
_{\sigma_{\Gamma}}) ^{\rm fs}$.
\end{corollary}
\begin{corollary} \label{theorem1OSMTsigma}
${\cal R_{\sigma} ^{\rm pr}}$ is (isomorphic to) $(\Gamma/\sim
_{\sigma_{\Gamma}}) ^*$.
\end{corollary} \mn
We call the order type of  $(\Gamma/\sim
_{\sigma_{\Gamma}})$ the {\bf rank} of the automorphism $\sigma_{\Gamma}$ .
We now can construct $\omega$-increasing automorphisms of arbitrary principal
difference rank. Corollary \ref{arbitrary} below, compared to Corollary \ref{omega} demonstrates
the discrepancy between the chains ${\cal  R_{\sigma} ^{\rm pr}}$ and ${\cal R}_{\sigma}\cap {\cal
R}^{\rm pr}$.
\begin{corollary}\label{arbitrary}
Given any order type $\tau$ there exists a maximally valued ordered
field endowed with an $\omega$-increasing automorphism of principal
difference rank $\tau$.
\end{corollary}
\begin{proof}
Set $\mu: = \tau^*$, and consider e.g. the linear ordering
$\Gamma:=\sum _\mu \Q$, that is, the concatenation of $\mu$ copies
of the non-negative rationals. Let $\ell$ be e.g. translation by $-1$ on $\Q$.
Define $\sigma _{\Gamma}$ to be the uniquely defined order
automorphism of $\Gamma$ extending $\ell$ on every copy. It is clearly a
strict left shift of rank $\mu$. Set e.g. $G:={\bf H}_{\Gamma} \R$. Then by Lemma \ref{lifting}
$\sigma_{\Gamma}$ lifts canonically to $\sigma_G$ on $G$. Now set
e.g. $K:= \R((G))$. By Lemma \ref{lifting}, Remark \ref{strict left shift} and Corollary  \ref{theorem1OSMTsigma}, $\sigma_G$ lifts canonically to an $\omega$-increasing
automorphism of $K$  of principal difference rank $\mu^*=\tau$.
\end{proof}
\begin{example} \label{k-s}
 Consider the chain $\Gamma= \Z \times \Z$ (the lexicographic product of two copies of $\Z$ ). We endow $\Gamma$  with the automorphisms  $\tau((x,
y)):=(x-1, y)$ and
$\sigma((x, y)):=(x, y-1)$. The rank of   $\tau$ is one and that of $\sigma$ is $\Z$. Both are strict left shifts. Lifting those automorphisms to  $G:={\bf H}_{\Gamma} \R$ and then to  $K:= \R((G))$ as in the proof of Corollary \ref{arbitrary}, we obtain $\omega$-increasing automorphisms of $K$ of distinct principal difference ranks.
\end{example}

For a regular uncountable cardinal ${\kappa}$,  let us denote by $G_{\kappa}$ the $\kappa$-bounded Hahn group, that is, the subgroup of $G={\bf H}_{\Gamma} \R$ consisting of elements with support of cardinality $< {\kappa}$.
Similarly,  we denote by $\R((G))_{\kappa}$ the $\kappa$-bounded Hahn field, i.e. the subfield of $K= \R((G))$  consisting of series with support of cardinality $< {\kappa}$. If $\kappa = \kappa^{< \kappa}$ then $\R((G_{\kappa}))_{\kappa}$ has cardinality $\kappa$, see \cite{[A-K]}.

We now generalize Example \ref{k-s}. In \cite[ Corollary 14]{[K-S]}, we construct for every  infinite cardinal $\kappa$ a chain $\Gamma$ of cardinality  $\kappa$ which admits of family of $2^{\kappa}$ strict left shift automorphisms, of pairwise distinct ranks.  Lifting those automorphisms  to $\R((G_{\kappa}))_{\kappa}$, we conclude as in \cite[ Theorem 9]{[K-S]}:
\begin{theorem} \label{last}
 Let  $\kappa = \kappa^{< \kappa}$ be a regular uncountable cardinal and $\Gamma$
be any chain of cardinality $\kappa$
which admits a family of  $2^{\kappa}$ strict left shift automorphisms of pairwise distinct ranks. Then the corresponding $\kappa$-bounded Hahn field  $\R((G_{\kappa}))_{\kappa}$ of cardinality $\kappa$ admits a family of  $2^{\kappa}$
$\omega$-increasing automorphisms of distinct principal difference ranks.
\end{theorem}


\end{document}